\documentclass[a4paper,12pt]{amsart}

\usepackage{amsmath, amssymb, amsthm}
\usepackage[alphabetic]{amsrefs}
\usepackage[utf8]{inputenc}
\usepackage[english]{babel}
\usepackage{color}

\newcommand{\NN}{\mathbb{N}}
\newcommand{\ZZ}{\mathbb{Z}}
\newcommand{\QQ}{\mathbb{Q}}

\newcommand{\CC}{\mathbb{C}}

\newcommand{\holo}{\mathcal{O}}

\DeclareMathOperator{\aut}{Aut}
\DeclareMathOperator{\enmo}{End}
\DeclareMathOperator{\Fix}{Fix}

\newtheorem{theorem}{Theorem}[section]
\newtheorem{lemma}[theorem]{Lemma}
\newtheorem{corollary}[theorem]{Corollary}
\newtheorem{proposition}[theorem]{Proposition}
\newtheorem{remark}[theorem]{Remark}

\theoremstyle{definition}
\newtheorem{definition}[theorem]{Definition}

\title{Reinhardt domains determined by their endomorphisms}

\author{Rafael B. Andrist}
\address{Faculty of Mathematics and Physics, 
University of Ljubljana, Ljubljana, Slovenia}
\email{rafael-benedikt.andrist@fmf.uni-lj.si}
\author{W{\l}odzimierz Zwonek}
\address{Institute of Mathematics, Faculty of Mathematics and Computer Science, Jagiellonian University, \L ojasiewicza 6, 30-348 Krak\'ow, Poland}
\email{wlodzimierz.zwonek@uj.edu.pl}

\keywords{Holomorphic self-mapping, endomorphisms, semigroup, pseudoconvex Reinhardt domain}
\subjclass[2020]{32Q02, 32E10, 20M20} 

\begin{document}

\begin{abstract}
We show that pseudoconvex Reinhardt domains in dimension two with isomorphic semigroups of holomorphic endomorphisms are biholomorphically or anti-biholomorphically equivalent. Moreover, we show that every Stein manifold that retracts to a properly embedded copy of the punctured complex line, is determined (up to biholomorphic or anti-biholomorphic equivalence) by its semigroup of holomorphic endomorphisms. 
\end{abstract}

\maketitle

\section{Introduction}

It is a classical result that Stein manifolds are determined by their rings of holomorphic functions -- up to (anti-)biholomorphic equivalence which goes back to a theorem of Bers in the case of one complex variable. Instead of functions, we consider the question whether the holomorphic endomorphisms, i.e.\ the holomorphic self-mappings, determine a complex manifold. The holomorphic endomorphisms of a complex manifold $X$ consist at least of all the constant mappings and the identity, and thus form a monoid which we denote by $\enmo(X)$. 

The following lemma is due to Schreier \cite{02519973}:
\begin{lemma}
\label{lem-conjugation}
Let $X$ and $Y$ be complex manifolds. Let $\Phi \colon \enmo(X) \rightarrow \enmo(Y)$ be an isomorphism of semigroups.
Then there exists a bijective mapping $\varphi \colon X \rightarrow Y$ such that \[\Phi(f) = \varphi \circ f \circ \varphi^{-1} \quad \forall \, f \in \enmo(X)\]
\end{lemma}
A variation of this Lemma with slightly different conditions can be found in the work of Buzzard and Merenkov \cite{MR2010321}. Instead of an abstract isomorphism of semigroups, we can therefore study the bijection $\varphi \colon X \to Y$ between complex manifolds and try to establish that $\varphi$ is (anti-)biholomorphic. The Lemma motivates the following definition:
\begin{definition}[see Andrist \cite{MR2763719}*{Definition 2.3}]
A (set-theoretic) mapping $\varphi \colon X \to Y$ between complex manifolds is called a \textit{conjugating mapping} if it is bijective and if it induces a homomorphism $\Phi \colon \enmo(X) \to \enmo(Y)$ of the endomorphism semigroups by conjugating $\enmo(X) \ni f \mapsto \varphi \circ f \circ \varphi^{-1} \in \enmo(Y)$. If, in addition, $\Phi$ is an isomorphism, then $\varphi$ is called an \textit{iso-conjugating mapping}.
\end{definition}

In the holomorphic category, the question of determining a complex manifold by its semigroup of endomorphisms has first been studied by Eremenko \cite{00370670} who established this result for Riemann surfaces that admit bounded holomorphic functions. His method crucially relies on first showing the continuity of $\varphi$ and on the theory about the solutions to the Schröder equation.

The case of the complex line $\CC$ had been studied before by Hinkkanen \cite{00126289} who investigated the possible automorphisms of $\enmo(\CC)$, i.e.\ he showed that an iso-conjugation $\varphi \colon \CC \to \CC$ is (anti-)biholomorphic. Merenkov \cite{01721089} extended Eremenko's methods and result to the case of bounded domains in $\CC^n$. Using a result of Buzzard and Hubbard \cite{01463210} about certain Fatou--Bieberbach domains that form a subbasis of the topology of $\CC^n$, Merenkov and Buzzard \cite{MR2010321} extended the result later also to $\CC^n$.

The first author generalized the result for $\CC^n$ to any Stein manifold that contains a properly embedded complex line $\CC$ \cite{MR2763719} by a new method that does not rely on establishing continuity a priori. The analogous result also holds for affine varieties over any algebraically closed field, as shown by Kraft and the first author \cite{06350076}.

In this article, we first study the one-dimensional case that had resisted for the longest time, the punctured complex line $\CC_\ast$:
\begin{proposition}
Let $\varphi \colon \CC_\ast \to \CC_\ast$ be a conjugating mapping. Then $\varphi$ is biholomorphic or anti-biholomorphic.
\end{proposition}
For the applications, it is important that we do not require $\varphi$ to be iso-conjugating in this proposition, but just conjugating.

Next, we cover the case of two-dimensional pseudoconvex Reinhardt domains:
\begin{theorem}\label{theorem-reinhardt}
Let $D_1, D_2 \subset \CC^2$ be pseudoconvex Reinhardt domains with an isomorphism $\Phi \colon \enmo(D_1) \to \enmo(D_2)$ of semi-groups. Then there exists a biholomorphic or anti-biholomorphic  $\psi \colon D_1 \to D_2$ such that $\Phi(f) = \psi \circ f \circ \psi^{-1}$. 
\end{theorem}

For unbounded domains, we can't drop the requirement of pseudoconvexity; we provide a concrete counterexample in Theorem \ref{thm:counterexample}.

In higher dimensions, we can prove the following general result:
\begin{theorem}\label{theorem-punctured-plane}
Let $X$ be a Stein manifold and let $i \colon \CC_\ast \to X$ be a proper holomorphic embedding such that $i(\CC_\ast)$ is a topological retract of $X$.
Let $\Phi \colon \enmo(X) \to \enmo(Y)$ be an isomorphism of semi-groups where $Y$ is a complex space. Then there exists a biholomorphic or anti-biholomorphic  $\varphi \colon X \to Y$ such that $\Phi(f) = \varphi \circ f \circ \varphi^{-1}$. 
\end{theorem}
In particular, this also implies that if $\enmo(Y) \cong \enmo(\CC_\ast)$, then $Y$ is (anti-)biholomorphic to $\CC_\ast$.

\section{Conjugating mappings for $\CC_\ast$}
\label{section-punctured-plane}

\begin{lemma}
Let $\varphi \colon \CC_\ast \to \CC_\ast$ be a conjugating mapping with $\varphi(1) = 1$. Then $\varphi$ is multiplicative, i.e.\ it is actually a group isomorphism of $\CC_\ast$.
\end{lemma}
\begin{proof}
The result is deduced by looking solely at the automorphism group of $\CC_*$, consisting of the mappings $z \mapsto a \cdot z^{\pm 1}, a \in \CC_\ast$. Therefore, with $w = \varphi^{-1}(z)$, we can write:
\begin{equation*}
\varphi(a w) = A(a) \cdot \left(\varphi(w)\right)^{(\pm 1)_a}
\end{equation*}
Since $\varphi(1) = 1$, it follows that in fact $A(a) = \varphi(a)$. The cases $(\pm 1)_a = -1$ can be excluded as follows:
The elements of order $2$ in $\aut(\CC_\ast)$ are $z \mapsto -z$ and $z \mapsto a/z$ for any $a \in \CC_\ast$. Hence it follows that 
\[
\varphi(a w) = \varphi(a) \cdot \varphi(w), \; \text{ for all } a \in \CC \setminus \{0, -1\}
\]
We only need to exclude that $\varphi(-w) = \varphi(-1)/\varphi(w)$.
For this we consider $z \mapsto i z$ whose second iterate is $z \mapsto -z$.
Since $\varphi(i w) = \varphi(i) \cdot \varphi(w)$, we can conclude that 
$\varphi(-w) = \varphi(-1) \cdot \varphi(w)$. It also follows $\varphi(-1) = -1$ and $\varphi(i) = \pm i$.
\end{proof}

\begin{lemma}
\label{lemma-cstar-exp}
Under the same assumptions, $\varphi$ conjugates $z \mapsto \exp(z)$ to $z \mapsto \exp(c z)$ or to $z \mapsto \exp(c/z)$  for some $c \in \CC_\ast$.
\end{lemma}
\begin{proof}
Let $p_n(z) := z^n$. Multiplicativity of $\varphi$ yields $\varphi \circ p_n \circ \varphi^{-1} = p_n$. 
Every holomorphic endomorphism of $\CC_\ast$ is of the form
\[
z \mapsto z^m \cdot \exp\left(\sum_{k=-\infty}^\infty c_k z^k\right), \quad m \in \ZZ
\]
Therefore, $z \mapsto \exp(z)$ is necessarily conjugated by $\varphi$ to the form above. Now, consider: 
\[
\varphi \circ p_n \circ \exp \circ \varphi^{-1} = p_n \circ \varphi \circ \exp \circ \varphi^{-1}
\]
The right-hand side can be written explicitly as:
\[
%z^m \cdot \exp\left(\sum_{k=-\infty}^\infty c_k z^k\right) = 
z^{nm} \cdot \exp\left(\sum_{k=-\infty}^\infty n c_k z^k\right)
\]
Using the multiplicativity of $\varphi$, the left-hand side can be written explicitly as:
\begin{align*}
\varphi \circ p_n \circ \exp \circ \varphi^{-1}(z)
&= \varphi(\exp(n \varphi^{-1}(z))) \\
&= \varphi(\exp(\varphi^{-1}( \varphi(n) \cdot z))) \\
&= (\varphi(n) \cdot z)^m \cdot \exp\left(\sum_{k=-\infty}^\infty c_k (\varphi(n) \cdot z)^k\right)
\end{align*}
Equating both sides, we now obtain:
\begin{align*}
m &= 0 \\
c_k \neq 0, k \neq 0 \Longrightarrow \varphi(n^k) &= n \\
c_0 \neq 0 \Longrightarrow 1 &= n \mod 2 \pi i \ZZ
\end{align*}
Since the mapping can't be constant, there exists exactly one $k \in \ZZ$ with $c_k \neq 0$, and, therefore, $\exp$ is conjugated to
\[
z \mapsto \exp(c_k z^k), \quad k \in \ZZ
\]
Let $\zeta \in \CC_\ast \setminus \{1\}$ be a primitive $k$-th root of unity and let $q(z) := \zeta z$. Such $q$ is conjugated to $\varphi(\zeta) z$ where $\varphi(\zeta)$ is also a $k$-th root of unity. We observe that $\exp \neq \exp \circ q$, but their conjugates are equal. Hence, $k = \pm 1$.
\end{proof}

\begin{corollary}
If we assume (after applying $z \mapsto 1/z$ if necessary) that $\varphi$ conjugates $z \mapsto \exp(z)$ to $z \mapsto \exp(c z)$ and if we set $\varphi(0) := 0$, then $\varphi$ is additive.
\end{corollary}
\begin{proof}
By multiplicativity of $\varphi$ we have:
\[
\varphi(\exp(a + b)) = \varphi(\exp(a)\cdot\exp(b)) = \varphi(\exp(a))\cdot\varphi(\exp(b)) 
\]
Applying Lemma \ref{lemma-cstar-exp} to both sides, we obtain:
\[
\exp(c\varphi(a + b))) = \exp(c\varphi(a))\cdot\exp(c\varphi(b)) = \exp(c(\varphi(a)+ \varphi(b)))
\]
This implies that $\varphi$ is additive modulo $\frac{2\pi i}{c} \ZZ$. Combining this with multiplicativity, it is easy to see that $\varphi$ is fully additive:
\[
\varphi(a+b) = \varphi(\varepsilon \cdot (a/\varepsilon + b/\varepsilon))
\]
The l.h.s.\ equals
\[\varphi(a) + \varphi(b) + \frac{2 \pi i}{c} k \quad \text{ for some } k \in \ZZ
,\] and the r.h.s.\ equals
\[
\varphi(a) + \varphi(b) + \varphi(\varepsilon) \frac{2 \pi i}{c} \ell \quad \text{ for some } \ell \in \ZZ.
\]
Choosing $\varepsilon$ s.t.\ $\varphi(\varepsilon)$ is irrational, requires $k = \ell  = 0$.
\end{proof}

\begin{proposition}
\label{prop-Cstar}
Such a mapping $\varphi$ is continuous, and hence holomorphic or anti-holomorphic.
\end{proposition}

\begin{proof}
Assume to get a contradiction that $\varphi$ is not continuous at $1$. Then there exists a sequence $(a_n)_n \subset \CC_\ast$ such that $\lim\limits_{n \to \infty} a_n = 1$ but $\lim\limits_{n \to \infty} \varphi(a_n) \neq 1$. Since $\varphi$ is additive and multiplicative, it equals the identity (if necessary, after complex conjugation $z \mapsto \overline{z}$) on $\QQ + i \QQ$. 

We may pass to a subsequence such that $\varphi(a_n)$ is either convergent with limit $A \in \CC_\ast$ or leaving any compact of $\CC_\ast$. We may assume that either $0 < |A| \leq 1$ or that $|\varphi(a_n)|$ tends to $\infty$, for otherwise we can replace the sequence $a_n$ by $1/a_n$. If  $|\varphi(a_n)|$ tends to $\infty$, we may choose a sequence of numbers $w_n \in \QQ + i \QQ$ such that $\lim_{n \to \infty} w_n = 0$ and $\lim_{n \to \infty} w_n \cdot (1 - \varphi(a_n)) = A$ with $A = 1/2$. Consider the sequence $b_n := 1 - w_n \cdot (1 - a_n)$. We have that $\lim_{n \to \infty} b_n = 1$ and $\lim_{n \to \infty} \varphi(b_n) = 1 -  \lim_{n \to \infty} w_n \cdot (1 - \varphi(a_n)) = 1/2$. We may assume that $0 < |A| \leq 1$ for otherwise we can replace $a_n$ by $b_n$.

We follow the last part of Hinkkanen's proof. In \cite{00126289} on page 153, he constructs a holomorphic function $f \colon \CC \to \CC$ with the following properties: Let $(z_n)_n \subset \CC$ be a sequence with $\lim_{n \to \infty} z_n = 0$. Then there exists such an $f$ with $f(0) = 0$, $f(\CC_\ast) \subseteq \CC_\ast$ and $f(z_n) \in \QQ + i \QQ$ for all $n \in \NN$.

We consider $((1 - a_k)/m)_{k, m}$. By a standard diagonalization argument, we obtain the sequence $(z_n)_n$ which contains $((1 - a_k)/m)_{k}$ as a subsequence for every $m$ and converges to $0$. 
We define $\widetilde{f}(z) := \exp(f(1 - z))$. Finally, we have that
\begin{align*}
\underbrace{\varphi \circ \widetilde{f} \circ \varphi^{-1}}_{\in \enmo(\CC_\ast)}\left(\varphi\left(1-\frac{1-a_n}{m}\right)\right) &= \varphi(\exp(f((1 - a_n)/m))) \\
 &= \exp(c\varphi(f(z_n))) = \exp(cf(z_n))
\end{align*}
for an appropriate subsequence of $z_n$ where $m$ is fixed. 
The r.h.s.\ converges to $1$, whereas the l.h.s.\ converges to $\varphi \circ \widetilde{f} \circ \varphi^{-1}\left(1-\frac{1 - A}{m}\right)$. Since $\varphi^{-1}(1) = 1$, this requires 
$f\left(\frac{1 - A}{m}\right) \in 2 \pi i \ZZ$ for every $m$. Either $f$ would therefore be unbounded near $0$ or identically equal to an integer multiple of $2 \pi i$, a contradiction. 

A continuous, additive and multiplicative $\varphi$ with $\varphi(i) = i$ equals the identity. 
\end{proof}

\section{Invariance under iso-conjugating mappings}
We present below a collection of invariance properties that are preserved under iso-conjugating mappings. These properties will be invoked later in the proof of Theorem~\ref{theorem-reinhardt} and Theorem~\ref{theorem-punctured-plane}.

First, for $F\in\enmo(X)$ we define 
\begin{math}
    \Fix(F):=\{x\in X:F(x)=x\}
\end{math} 
and call it {\it the set of fixed points} of the endomorphism $F$.

For iso-conjugated endomorphisms $F\in\enmo(X)$ and $G\in\enmo(Y)$ we have the equality $G\circ\psi=\psi\circ F$ or $G=\psi\circ F\circ \psi^{-1}$.

We get the following 
\begin{lemma}\label{lemma:fixed-points}
    Let $\psi \colon X\to Y$ be an iso-conjugating mapping conjugating  $F\in\enmo(X)$ and $G\in\enmo(Y)$. Then
\begin{equation}
    \psi(\Fix(F))=\Fix(G).
\end{equation}
\end{lemma}
\begin{proof} Actually, it is sufficient to show one inclusion, so let $x\in \Fix(F)$. Then
\begin{equation}
    G(\psi(x))=\psi(F(x))=\psi(x),
\end{equation}
so $\psi(x)\in\Fix G$.
\end{proof}

Similarly, iso-conjugating mappings leave the retracts invariant. Recall that an endomorphism $r\in\enmo(X)$ is called {\it a retraction} if $r\circ r=r$. Then $M:=r(X)$ is called {\it a retract}.

\begin{lemma}
\label{lemma-retract}
Let $\psi \colon X\to Y$ be an iso-conjugating mapping and let $M$ be a retract in $X$. Then $\psi(M)$ is a retract in $Y$.
\end{lemma}
\begin{proof}
    Let $r \colon X\to M$ be a holomorphic retraction. Then the holomorphic endomorphism $\widetilde{r}:=\psi\circ r\circ\psi^{-1}$ of $Y$ is a retraction with $\widetilde{r}(Y)=\psi(r(X))=\psi(M)$.
\end{proof}

It is evident that the property of having the same number of preimages of an endomorphism is an invariant of iso-conjugating mappings.

\begin{lemma}
    Let $\psi \colon X\to Y$ be an iso-conjugating mapping, where $X,Y$ have the same dimension. Let $F\in\enmo(X)$ and $G\in\enmo(Y)$ be two iso-conjugated (by $\psi$) endomorphisms and fix a positive integer $k$. 
    
    Then for every $x\in X$ $\#F^{-1}(x)=k$ iff for every $y\in Y$ $\#G^{-1}(y)=k$.
\end{lemma}

\section{Manifolds containing $\CC_\ast$ as a retract}

\begin{proposition}
\label{prop-key}
Let $X$ be a complex manifold and let $Y$ be any complex space. 
Assume there exists a proper holomorphic embedding $i \colon Z \to X$ of a complex manifold $Z$ such that the following holds:
\begin{enumerate}
\item \label{property1} $\forall f \in \holo(i(Z), i(Z)) \; \exists F \in \holo(X, X) \,:\, F|i(Z) = f$ and $F(X) = i(Z)$.
\item \label{property2} $\exists G_1, \dots, G_q \in \holo(X, X), \; \exists x_1, \dots, x_q \in X \;:$
\[i(Z) = \bigcap_{j=1}^{q} G_j^{-1}(x_j)\]
\item \label{property3} For any complex space $W$ and any conjugating mapping $\varphi \colon Z \to W$ we require that $Z$ is biholomorphic or anti-biholomorphic to $W$.
\item \label{property4} There exists a proper holomorphic embedding $H \colon X \to Z^n$.
\end{enumerate}
Let $\Phi \colon \enmo(X) \to \enmo(Y)$ be an isomorphism of semi-groups. Then there exists a biholomorphic or anti-biholomorphic $\varphi \colon X \to Y$ such that $\Phi(f) = \varphi \circ f \circ \varphi^{-1}$. 
\end{proposition}

\begin{proof}
By Lemma \ref{lem-conjugation} there exists a bijective mapping $\varphi \colon X \to Y$ such that
$\Phi(f) = \varphi \circ f \circ \varphi^{-1}$. Denote $A := i(Z)$ and $B = \varphi(A)$. 
By Property \ref{property2} we have that
\[
B = \varphi(i(Z)) = \bigcap_{j=1}^{q} \varphi \circ G_j^{-1} \circ \varphi^{-1}( \varphi(x_j))
\]
and since $\varphi \circ G_j^{-1} \circ \varphi^{-1}$ is holomorphic, $B$ is an analytic subset of $Y$ and hence a complex space. By Property \ref{property1}, we finally obtain every holomorphic self-mapping of $A$ as the restriction a holomorphic self-mapping of $X$ which is by assumption conjugated to a holomorphic self-mapping of $Y$ that restricts to a self-mapping of $B$. Note that we do a priori not know that the restriction $\varphi | A$ is iso-conjugating, only that it is conjugating even though $\varphi$ is iso-conjugating. Applying Property \ref{property3} to $W = B$, we conclude that $B$ is biholomorphic or anti-biholomorphic to $Z$. Without loss of generality, we may assume that $\varphi \circ i \colon Z \to W$ is biholomorphic, for otherwise we just replace the whole $X$ by its complex conjugate.

For $f \in \holo(X, Z)$ we consider the mapping 
\[
f \mapsto (\varphi \circ i)^{-1} \circ (\varphi \circ (i \circ f) \circ \varphi^{-1}) = f \circ \varphi^{-1}
\]
which is in $\holo(Y, Z)$ as the composition of holomorphic mappings. 

Let $H = (h_1, \dots, h_n) \colon X \to Z^n$ be the proper holomorphic embedding that exists according to Property \ref{property4}. Then $H \circ \varphi^{-1} = (h_1 \circ \varphi^{-1}, \dots, h_n \circ \varphi^{-1}) \colon Y \to Z^n$ is now also a holomorphic injection. 

By definition, the images of $H$ and $H \circ \varphi^{-1}$ coincide. We conclude that $\varphi^{-1} = H^{-1} \circ (H \circ \varphi^{-1}) \colon Y \to X$ is biholomorphic. 
\end{proof}

\begin{remark}
\label{rem-CtoCstar}
If $h = (h_1, \dots, h_n) \colon X \to \CC^n$ is an injective holomorphic mapping, then 
$H := (e^{h_1}, e^{\sqrt{-1} h_1}, \dots, e^{h_n}, e^{\sqrt{-1} h_n}) \colon X \to (\CC^{\ast})^{2n}$ is also an injective holomorphic mapping. If $h$ is proper, then so is $H$.
\end{remark}

\begin{proof}[Proof of Theorem \ref{theorem-punctured-plane}]
We apply Proposition \ref{prop-key} with $Z = \CC_\ast$.
\begin{enumerate}
\item Every continuous mapping from the Stein manifold $X$ to the Oka manifold $i(\CC_\ast)$ is homotopic to a holomorphic mapping from $X$ to $i(\CC_\ast)$ and moreover, the values on the subvariety $i(\CC_\ast) \subset X$ can be prescribed, see \cite{MR3700709}*{Corollary 5.4.5} in the monograph of Forstneri\v{c}. Hence, being a topological retract of $X$ is equivalent to being a holomorphic retract of $X$ for $i(\CC_\ast)$. 
\item Since $X$ is Stein, by Remark \ref{rem-CtoCstar} every analytic subset of $X$ can be written as a finite intersection of pre-images of holomorphic mappings $f_j \colon X \to \CC_\ast$, hence as a finite intersection of pre-images of holomorphic self-mappings $i \circ f_j \colon X \to X$.
\item By assumption, every holomorphic automorphism of $A := i(\CC_\ast)$ is the restriction of a holomorphic self-mapping of $X$. Hence, $\CC_\ast$ acts transitively on $B := \varphi(A)$ which is an analytic subset of $B$, and therefore a smooth complex manifold.

$B$ is connected: Let $F \colon X \to A$ be a holomorphic retraction. Assume to get a contradiction that $G \colon B \to B$ is such that $G$ equals the identity on one connected component of $B$ but is constant on all other connected components of $B$. By Lemma \ref{lemma-retract}, $G \circ \varphi \circ F \circ \varphi^{-1} \colon Y \to B$ is a holomorphic self-mapping that we can conjugate back to a holomorphic self-mapping $\varphi^{-1} \circ G \circ \varphi \circ F \colon X \to A$ which restricts to a non-constant holomorphic self-mapping $A \to A$ that omits an infinite number of points which contradicts Picard's theorem.

Next, we push forward the group structure of $\CC_\ast$ to $B$. Again the assumption implies that left- and right-multiplication on $B$ are holomorphic. By Hartogs's separate analyticity theorem and by the implicit function theorem, $B$ is a complex Lie group with its holomorphic structure induced by $Y$. Since $B$ is abelian, it is isomorphic to $\CC^n \times (\CC_\ast)^m \times T$ where $T$ is a toroidal group. By considering the number of group elements of finite order, we can exclude all factors but $\CC_\ast$.

Finally, we can apply Proposition \ref{prop-Cstar} to $\varphi|A \colon A \to B$.
\item Since $X$ is Stein, by Remark \ref{rem-CtoCstar} we can embed it properly holomorphically into $(\CC_\ast)^N$ for $N$ large enough. \qedhere
\end{enumerate}
\end{proof}

\section{Pseudoconvex Reinhardt domains in higher dimensions} 

\subsection{Geometry of pseudoconvex Reinhardt domains} Before we go on to the proof of Theorem~\ref{theorem-reinhardt} and show how the two-dimensional pseudoconvex Reinhardt domains are determined by their endomorphisms, we recall some basic notions and facts on pseudoconvex Reinhardt domains in arbitrary dimensions. As a good reference for the properties of Reinhardt domains that we need in the sequel we refer the reader to monographs \cite{MR2396710} and \cite{MR4201928} by Jarnicki and Pflug.

A domain $D\subset\mathbb C^n$ is called {\it Reinhardt} if $D$ is invariant under rotations in each coordinate separately, i.e.\ if $z\in D$ then $(\omega_1z_1,\ldots,\omega_nz_n)\in D$, $|\omega_j|=1$. Then, it makes then sense to define {\it the logarithmic image} of $D$
\begin{equation}
    \log D:=\{x\in\mathbb R^n:(e^{x_1},\ldots,e^{x_n})\in D\}.
\end{equation}
We also define $V_j:=\mathbb C^{j-1}\times\{0\}\times\mathbb C^{n-j}$, $j=1,\ldots,n$.

The pseudoconvexity of $D$ is completely understood. More precisely, 
the characterization may be given in the following inductive way: A Reinhardt domain $D$ is pseudoconvex iff $\log D$ is convex and for all $j=1,\ldots,n$ the intersection $D\cap V_j$ is either empty or is a pseudoconvex domain (as a subset of $\mathbb C^{n-1}$).

We shall often rely on the following characterization of hyperbolic pseudoconvex Reinhardt domains (see Theorem 1.1 and 1.2 in \cite{MR1678013}, see also \cite{MR1785672}).
\begin{theorem}\label{theorem:hyperbolic-reinhardt}
    Let $D$ be a pseudoconvex Reinhardt domain. Then the following are equivalent:
    \begin{itemize}
        \item $D$ is (Carath\'eodory or Kobayashi) hyperbolic,
        \item $\log D$ contains no real line and $D\cap V_j$ is either empty or is (Carath\'eodory or Kobayashi) hyperbolic $j=1,\ldots,n$,
        \item $D$ is biholomorphic to a bounded Reinhardt domain.
    \end{itemize}
\end{theorem}
Let us also recall that the biholomorphism to a bounded domain in Theorem~\ref{theorem:hyperbolic-reinhardt} may be realized as an algebraic mapping (all the coordinates will be monomials).

\subsection{Determining of Reinhardt domains in several variables}
Before we concentrate on the complete solution of the two-dimensional situation, we shall present how the problem of determining pseudoconvex Reinhardt domains may be solved partially in higher dimensions.

\begin{remark}\label{remark-reinhardt} Following the characterization of hyperbolic pseudoconvex Reinhardt domains given in Theorem~\ref{theorem:hyperbolic-reinhardt}
when working with domains not biholomorphic to bounded ones, we may consider two cases:

\begin{itemize}
    \item $D$ is complete ($0\in D$) and in this case some coordinate axis $\mathbb C e_j\subset D$, which gives a retract by embedding $\mathbb C\owns\lambda\to \lambda e_j\in D$;
\item $0\not\in D$ and then the intersection $G_k$ of $D$ with some (possibly) lower dimensional $\mathbb C^k\times\{0\}^{n-k}$ (up to a permutation of variables)  is such  that the logarithmic image of $G_k$ contains a real line. 
\end{itemize}

Note that in the latter case the projection of $D$ on $\mathbb C^k\times\{0\}^{n-k}$ is $G_k$ so the retracts in $G_k$ are retracts in $D$. In other words: when $0\not\in D$ we are able to reduce the problem to the situation when $D\subset \mathbb C^n_*$ and $\log D$ contains a real line.
\end{remark}

It turns out that the existence of a rational real line in $\log D$, i.e.\ $a\in\log D$, $v\in \mathbb Z^n\setminus\{0\}$ (with the gcd of $v_j$, $j=1,\ldots,n$ equal to $1$) with $a+\mathbb Rv\subset\log D$ will make it possible to apply Theorem~\ref{theorem-punctured-plane}. Actually, without loss of generality, we may assume $a=(0,\ldots,0)$. Then we get the embedding of $\mathbb C_*$ in $D$ defined as follows:
\begin{equation}
i \colon \mathbb C_*\owns\lambda \mapsto (\lambda^{v_1},\ldots,\lambda^{v_n})\in D
\end{equation}
and the retraction mapping of $i(\mathbb C_*)$:
\begin{equation}
    r \colon D\owns z \mapsto i(z_1^{k_1}\ldots z_n^{k_n}) \in D,
\end{equation}
    where $k_j\in\mathbb Z$ and $k_1v_1+\ldots+k_nv_n=1$.

The above observation together with Remark~\ref{remark-reinhardt}, results of \cite{MR2763719}*{Theorem 3.3} and Theorem~\ref{theorem-punctured-plane} let us conclude the following result.

\begin{corollary}\label{corollary:higher-dimension} Let $D$ be a pseudoconvex Reinhardt domain in $\mathbb C^n$ such that $0\in D$ or the logarithmic image of some (possibly) lower dimensional intersection $D\cap V$, where $V$ is a $k$-dimensional space spanned by vectors of the canonical basis $e_{j_1},\ldots,e_{j_k}$, $1\leq k\leq n$, $1\leq j_1<\ldots<j_k\leq n$, contains a rational line. Let $\Phi\colon\enmo(D)\to\enmo(Y)$ be an isomorphism of semi-groups where $Y$ is a complex space. Then there exists a biholomorphic or anti-biholomorphic $\varphi \colon D\to Y$ such that $\Phi(f)=\varphi\circ f\circ \varphi^{-1}$.    
\end{corollary}

By results of \cite{01721089}, for obtaining Theorem~\ref{theorem-reinhardt} we may restrict ourselves to the situation when $D_1$ and $D_2$ are not biholomorphic to bounded domains. Additionally, by Corollary~\ref{corollary:higher-dimension}, we may restrict ourselves to the case when $D\subset\mathbb C_*^2$ and $\log D$ contains only real lines of irrational type. In this situation we have two holomorphically inequivalent classes of domains to deal with. This follows from \cite{05208635}, where the complete description of holomorphic equivalence of unbounded pseudoconvex Reinhardt domains is given. These domains (defined later) are denoted $D_{\gamma}^*$ and $D_{\gamma,r}$ -- their logarithmic images will be either half-planes or strips. We introduce them in the next subsection (providing more, namely three types of domains). 

\subsection{Holomorphic endomorphisms of special irrational Reinhardt domains in $\mathbb C^2$}
By the earlier considerations to prove Theorem~\ref{theorem-reinhardt} we are reduced to examining  two classes of domains in $\mathbb C_*^2$ which we define below. The description of their endomorphisms plays a key role in the proof. 

The three special classes of pseudoconvex Reinhardt domains in $\mathbb C^2$ having an irrational line in the logarithmic image are defined below.

For $\gamma>0$, 
$\gamma\not\in\mathbb Q$ denote a domain of type (I) (or {\it complete half-plane type}) as follows
\begin{equation}
D_{\gamma}:=\{z\in\mathbb C^2:|z_1||z_2|^{\gamma}<1\}
\end{equation}
and the domain of type (II) (or {\it non-complete half-plane type}) is of the form: 
$D_{\gamma}^*:=D_{\gamma}\cap \mathbb C_*^2$. Additionally, for $r>1$ we define the domain of type (III) (or {\it strip type}) as
\begin{equation}
    D_{\gamma,r}:=\{z\in\mathbb C^2:1/r<|z_1||z_2|^{\gamma}< r\}.
\end{equation}
Recall that all the domains $D_{\gamma}$, $D_{\gamma}^*$ and $D_{\gamma,r}$ are pairwise (for different irrational $\gamma$ and $r>1$) biholomorphically inequivalent (see \cite{05208635}).

In our study, a special role is played by the mapping
\begin{equation}
    p_{\gamma}(\lambda):=(\exp(-\gamma\lambda),\exp(\lambda)),\;\lambda\in\mathbb C.
\end{equation}
We also denote $V(\gamma):=p_{\gamma}(\mathbb C)$ and $V_{\rho}(\gamma):=\{z\in\mathbb C^2:|z_1||z_2|^{\gamma}=\rho\}$, $\rho>0$. Note that $V(\gamma)$ is dense in $V_1(\gamma)$.

For $w,z\in\mathbb C^2$ we denote
\begin{equation}
    w\cdot z:=(w_1z_1,w_2z_2).
\end{equation}
For $z\in\mathbb C_*^2$, $n\in\mathbb Z$ we denote
\begin{equation}
    z^n:=(z_1^n,z_2^n).
\end{equation}

In $\mathbb C^2_*$ we introduce the equivalence relation:
\begin{equation}
    w\sim_{\gamma} z \iff \text{ there is a $\lambda\in\mathbb C$ such that } \left(\frac{w_1}{z_1},\frac{w_2}{z_2}\right)=p_{\gamma}(\lambda).
\end{equation}

For $a\in\mathbb C_*^2$ we have
\begin{equation}
[a]_{\sim_{\gamma}}=a\cdot p_{\gamma}(\mathbb C).
\end{equation}

\begin{theorem}\label{theorem:endomorphisms}
$\enmo(D_{\gamma,r})$ comprises the following mappings:

\begin{equation}
    D_{\gamma,r}\owns z \mapsto (z_1^{k_2}z_2^{k_1},z_1^{l_2}z_2^{k_1})\cdot a\cdot p_{\gamma}(h(z))\in D_{\gamma,r}
\end{equation}
where $h\in\mathcal O(D_{\gamma,r})$, $\alpha\gamma=k_1+\gamma l_1$, $\alpha=k_2+\gamma l_2$, $k_j,l_j\in\mathbb Z$, $\alpha\in[-1,1]$. In the case $\alpha=0$ (and then $k_j=l_j=0$) $a\in D_{\gamma,r}$. In the case $\alpha\neq 0$, $|a_||a_2|^{\gamma} r^{\pm \alpha}\in[1/r,r]$.

$\enmo(D_{\gamma}^*)$ comprises the following mappings:

\begin{equation}
    D_{\gamma}^*\owns z \mapsto (z_1^{k_2}z_2^{k_1},z_1^{l_2}z_2^{l_1})\cdot a\cdot p_{\gamma}(h(z))\in D_{\gamma}^*
\end{equation}
where $h\in\mathcal O(D_{\gamma}^*)$, $\alpha\gamma=k_1+\gamma l_1$, $\alpha=k_2+\gamma l_2$, $\alpha\geq 0$, $k_j,l_j\in\mathbb Z$. In the case $\alpha=0$, $a\in D_{\gamma}^*$. In the case $\alpha>0$,  $0<|a_1||a_2|^{\gamma}\leq 1$.

$\enmo(D_{\gamma})$ comprises the following mappings:
\begin{equation}
    D_{\gamma}\owns z\mapsto (z_1^{k_2}z_2^{k_1},z_1^{l_2}z_2^{l_1})\cdot a\cdot p_{\gamma}(h(z))\in D_{\gamma},
\end{equation}

where $h\in\mathcal O(D_{\gamma})$, $\alpha\gamma=k_1+\gamma l_1$, $\alpha=k_2+\gamma l_2$, $\alpha\geq 0$, $k_j,l_j\in\mathbb N$

and the mappings
\begin{equation}
    D_{\gamma}\owns z\mapsto (h(z),0)\in D_{\gamma},
\end{equation}

\begin{equation}
    D_{\gamma}\owns z\mapsto (0,h(z))\in D_{\gamma},
\end{equation}
where $h\in\mathcal O(D_{\gamma})$.

In the case $\alpha=0$, $a\in D_{\gamma}$. In the case $\alpha>0$,  $0<|a_1||a_2|^{\gamma}\leq 1$.

\end{theorem}

In the situation described in the theorem above, the only $\alpha$ such that $\alpha\gamma=k_1+\gamma l_1$, $\alpha=k_2+\gamma l_2$ is called {\it the degree} of the endomorphism. In the case of a strip type domain, the degree $\alpha\in[-1,1]$. In the case of domains of complete (and non-complete) half-plane types, the degree $\alpha\geq 0$. 

   Direct calculations also show that for the endomorphism $F$ of degree $\alpha$ we get 
   %(in the case $\alpha=0$ we put $b$ in the place of $a$)
    \begin{equation}
        |F_1(z)||F_2(z)|^{\gamma}=|a_1||a_2|^{\gamma}(|z_1||z_2|^{\gamma})^{\alpha}
    \end{equation}
    for all $z\in D_{\gamma,r}$ (respectively, $z\in D_{\gamma}^*$).

Consequently, we see that for the given endomorphism of degree $\alpha$ we find a positive $\theta$ such that any set $V_{\rho}(\gamma)$ is mapped by the mapping to $V_{|a_1||a_2|^{\gamma} \rho^{\alpha}}(\gamma)$ %(with $a$ replaced by $b$ for $\alpha=0$). 
In the case of order $\alpha=0$ the image under the endomorphism is contained in $a\cdot p_{\gamma}(\mathbb C)$.
%for some $b\in D_{\gamma,r}$ (respectively, $D_{\gamma}^*$). 
In the case $\alpha\neq 0$ the image of the endomorphism is not contained in $V_{\rho}(\gamma)$ for any $\rho$.

Recall that the description of biholomorphic equivalence of domains studied above was discussed in \cite{MR1144475}, \cite{MR1189968}. The description of proper holomorphic mappings between the domains as above may be found in \cite{05208635}. The method of the proof of Theorem~\ref{theorem:endomorphisms} follows the ideas developed in \cite{05208635}.

\begin{remark} It follows from \cite{05208635} that the holomorphic automorphisms of $D_{\gamma,r}$ are precisely the endomorphisms
  \begin{equation}
      D_{\gamma,r}\owns z \mapsto a\cdot z^{\pm 1}\in D_{\gamma,r},\;|a_1||a_2|^{\gamma}=1.
  \end{equation}
  In particular, the involutive automorphisms are $z\to a\cdot z^{-1}$, $|a_1|a_2|^{\gamma}=1$ and the mappings $z\to (\pm z_1,\pm z_2)$.

  On the other hand, using \cite{05208635} once more (and the characterization of algebraic automorphisms of $\mathbb C_*^n$ from \cite{MR1678013}), we get that automorphisms of $D_{\gamma}^*$ are
  \begin{equation}
      D_{\gamma}^*\owns z \mapsto a\cdot (z_1^{k_2}z_2^{k_1},z_1^{l_2}z_2^{l_1})\in D_{\gamma}^*,\; |a_1||a_2|^{\gamma}=1,
  \end{equation}
  where $\alpha\gamma=k_1+\gamma l_1$, $\alpha=k_2+\gamma l_2$, $\alpha>0$ and $|k_2l_1-k_1l_2|=1$.

In particular, the involutive holomorphic automorphisms of $D_{\gamma}^*$ are only $D_{\gamma}^*\owns z\to (\pm z_1,\pm z_2)\in D_{\gamma}^*$.
\end{remark}

\begin{proof}[Proof of Theorem~\ref{theorem:endomorphisms}]

Let $F$ be an endomorphism of one of domains as in the theorem.

First note that (in all three cases) for any $\rho\in(1/r,r)$ (respectively, $\rho\in (0,1)$) the function
\begin{equation}
    \mathbb C\owns\lambda \mapsto \log(|F_1(\rho p_{\gamma}(\lambda))||F_2(\rho p_{\gamma}(\lambda))|^{\gamma})
\end{equation}
is (sub)harmonic and bounded from above, so it is constant. As the set $(\rho,\rho)\cdot p_{\gamma}(\mathbb C)$ is dense in $V_{\rho^{1+\gamma}}(\gamma)$, the continuity of the mapping above implies the following
\begin{equation}
|F_1(z)||F_2(z)|^{\gamma} \text{ is constant on $V_{\rho^{1+\gamma}}(\gamma)$}.
\end{equation}
If for some $\rho$ the above constant would be zero (which is only possible in the case of $D_{\gamma}$) then, by the identity principle $F_1$ or $F_2$ would be identically $0$. Then there would be two possibilities for the endomorphisms:
$(0,h(z))$, $z\in D_{\gamma}$ or $(h(z),0)$, $z\in D_{\gamma}$, where $h\in\mathcal O(D_{\gamma})$. This means that we may assume that for all $z\in D_{\gamma}^*$ we get
%(respectively, $z\in D_{\gamma,r}$) 
$F(z)\in D_{\gamma}^*$.
%(respectively, $F(z)\in D_{\gamma,r}$). 

Therefore, we may now restrict our further considerations to the case of $D_{\gamma}^*$ (respectively, $D_{\gamma,r}$).

Consequently, the function (defined for $0<|\lambda|<1$ or $1/r<|\lambda|^{1+\gamma}<r$ with the values lying in $(-\infty,0)$ or $(-\log r,\log r)$)
\begin{equation}
    f \colon \lambda \mapsto \log(|F_1(\lambda,\lambda)||F_2(\lambda,\lambda)|^{\gamma})
\end{equation}
is harmonic and radial (i.e.\ $f(\lambda)=f(|\lambda|)$). 

Therefore, it is of the form $c(1+\gamma)\log|\lambda|+d$. Then
\begin{equation}    |F_1(z)||F_2(z)|^{\gamma}=\beta(|z_1||z_2|^{\gamma})^{\alpha},
\end{equation}
where $z$ is from $D_{\gamma}^*$ (respectively, $D_{\gamma,r}$).

As $F(D_{\gamma}^*)\subset D_{\gamma}^*$ (respectively, $F(D_{\gamma,r})\subset D_{\gamma,r}$) we get $\alpha>0$ and $0<\beta\leq 1$ or $\alpha=0$ and $0<\beta<1$ (respectively, $\alpha\neq 0$ and $\beta r^{\pm\alpha}\in[1/r,r]$ or $\alpha=0$ and $\beta\in(r,1/r)$).

Taking the special points $(1,\lambda)\in D_{\gamma}^*$ or $D_{\gamma,r}$, we get
\begin{equation}
    |f_1(\lambda)|^2|f_2(\lambda)|^{2\gamma}:=|F_1(1,\lambda)|^2|F_2(1,\lambda)|^{2\gamma}=\beta^2 |\lambda|^{2\gamma\alpha}
\end{equation}
for $0<|\lambda|<1$ (or $1/r<|\lambda|^{\gamma}<r$).

Taking the logarithm and differentiating w.r.t.\  $\lambda$ we get
\begin{equation}
\frac{f_1^{\prime}(\lambda)}{f_1(\lambda)}+\gamma \frac{f_2^{\prime}(\lambda)}{f_2(\lambda)}=\gamma\alpha\frac{1}{\lambda}
\end{equation}
for $0<|\lambda|<1$ (or $1/r<|\lambda|^{\gamma}<r$).

Taking the integral over the circle $|\lambda|=\rho$ (where $C_{\rho}$ denotes the parametrization of the circle) and keeping in mind that the indices of respective curves ($f_j\circ C_{\rho}$) are integers, we get the following property
\begin{equation}
    \alpha\gamma\in\mathbb Z+\gamma\mathbb Z.
\end{equation}
Repeating the same reasoning with $f_j(\lambda):=F_j(\lambda,1)$ we get the property
\begin{equation}
    \alpha\in\mathbb Z+\gamma\mathbb Z.
\end{equation}
In other words, we need to have
\begin{align}
  \alpha\gamma&=k_1+\gamma l_1,\\
  \alpha&=k_2+\gamma l_2
\end{align}
for some $k_1,l_1,k_2,l_2\in\mathbb Z$.

Then we get for $z\in D_{\gamma}^*$ (respectively, for $z\in D_{\gamma,r}$)
\begin{equation}
 \left|\frac{F_1(z)}{z_1^{k_2}z_2^{k_1}}\right|=\beta\left|\frac{z_1^{l_2}z_2^{l_1}}{F_2(z)}\right|^{\gamma}.
\end{equation}

We may use Lemma 9 in \cite{05208635} to get that 
\begin{equation}
    F(z)=(\beta_1 z_1^{k_2}z_2^{k_1}\exp(-\gamma h(z)),\beta_2 z_1^{l_2}z_2^{l_1}\exp(h(z)),
    \end{equation}
    where $h\in\mathcal O(D_{\gamma}^*)$ (respectively,  $h\in\mathcal O(D_{\gamma,r})$) with $\alpha$ and $\beta:=|\beta_1||\beta_2|^{\gamma}$ satisfying the earlier relations which gives the desired formulas.
    Direct calculations show that all the functions appearing in the statement of the theorem are actually endomorphisms.   
  %  , $0<|\beta_1||\beta_2|^{\gamma}\leq 1$, $\alpha\geq 0$ (respectively, $1/r\leq |\beta_1||\beta_2|^{\gamma}\leq r$, $\alpha\in\{-1,0,1\}$) and it gives a complete description of $\operatorname{End}(D_{\gamma}^*)$ (and that of $\operatorname{End}(D_{\gamma})$. Note that in the case $\alpha=0$ we need additionally to assume that $|\beta_1||\beta_2|^{\gamma}<1$ (respectively, $1/r<|\beta_1||\beta_2|^{\gamma}<r$).
\end{proof}

\begin{remark} If we assume additionally that $\gamma$ is not algebraic (even more generally, it is not a solution of the second degree equation with integer coefficients) then $k_2=\alpha$, $k_1=0$, $l_2=0$, $l_1=\alpha$.

Note that in the case $\alpha\in\mathbb Z$ we get that $\alpha=k_2=l_1$ and $k_1=l_2=0$. 
%Note also that in the case of the domain $D_{\gamma,r}$ we have $\alpha\in\{-1,0,1\}$.
\end{remark}

As already mentioned, the complete description of holomorphic automorphisms or proper holomorphic self-mappings of domains of type (I), (II) and (III) was found in \cite{05208635}. The proper holomorphic self-mappings of the above domains are always of some order $\alpha \neq 0$ with $h$ being constant. Note that there are no non-trivial (i.e.\ not automorphisms) proper holomorphic self-mappings of domains of strip type. On the other hand, non-trivial proper holomorphic self-mappings of domains of complete and non-complete half plane type do exist (for instance the mappings of the form $z^n$, $n=2,3,\ldots$) and they all have the property that the preimage of an arbitrary element of the domain contains the same number of elements (equal to the multiplicity of the proper mapping).

\section{Properties of endomorphisms of Reinhardt domains}
First we make a remark on involutive automorphisms.

\begin{remark} If $\psi \colon X\to Y$ is an iso-conjugating mapping then the involutive automorphisms are mapped by the conjugation w.r.t.\  $\psi$ to involutive automorphisms. 

%It follows from the description of holomorphic automorphisms of %$D_{\gamma,r}$ 
Recall that the involutive automorphisms of $D_{\gamma,r}$ are mappings $a\cdot z^{-1}$, $|a_1||a_2|^{\gamma}=1$ and three automorphisms of the form $(\pm z_1,\pm z_2)$.

On the other hand, involutive automorphisms of $D_{\gamma}^*$ comprise only the mappings $(\pm z_1,\pm z_2)$.

Consequently, by comparing the number of involutive automorphisms we then conclude that the domains $D_{\gamma}^*$ are not iso-conjugated with domains $D_{\gamma,r}$ or, equivalently, iso-conjugation leaves the type of domains invariant.
\end{remark}

%We make then a remark on a possible iso-conjugacy of Reinhardt domains of different types (that of $D_{\gamma}$ and $D_{\gamma,r}$). First note that the type of such a mapping cannot be distinct. Actually, this follows from the fact that Reinhardt domains $D_{\gamma}^*$ have non-proper proper holomorphic mappings (having the same number in any preimage) whereas the domains $D_{\gamma,r}$ do not have this property. Alternatively, that could be seen by comparing (the number) of involutive automorphisms).
Below, we shall use the fact that one-dimensional analytic sets cannot lie in $V_{\rho}(\gamma)$ for fixed $\gamma$. This fact may be well-known to specialists in complex analytic geometry, but since we could not find a reference, we present its short proof below.

\begin{lemma}\label{lemma:one-dimensional-analytic-set}
    There is no one-dimensional analytic set $M$ in the pseudoconvex domain $D\subset\mathbb C^2$ with $V_{\rho}(\gamma)$ such that $M\subset V_{\rho}(\gamma)\subset D$.
\end{lemma}
\begin{proof}
    Suppose the contrary. By the pseudoconvexity of $D$, there are holomorphic functions $f_j \colon D\to\mathbb C$, $j=1,\ldots,k$ such that $M=\bigcap\limits_{j=1}^k f_j^{-1}(0)$. Additionally, by applying Lemma 9 from \cite{05208635} at regular points $z_0$ of $M$, we get that $M$ is given near $z_0$ as the image of a mapping $a\cdot p_{\gamma}(\lambda)$ for $\lambda$ from some open neighborhood $U$ of $\lambda_0\in\mathbb C$. As the functions $f_j\circ(a\cdot p_{\gamma}):\mathbb C\to\mathbb C$ are holomorphic and vanish on $U$ they all vanish on $\mathbb C$. Consequently, $a\cdot p_{\gamma}(\mathbb C)\subset M$. But the closure of $a\cdot p_{\gamma}(\mathbb C)$ is $V_{|a_1||a_2|^{\gamma}}(\gamma)$, so $V_{|a_1||a_2|^{\gamma}}(\gamma)\subset M$ -- a contradiction.   
\end{proof}

\begin{remark}\label{remark:fixed-points-analytic-set} Looking at the form of holomorphic endomorphisms of $D_{\gamma,r}$ and $D_{\gamma}^*$ and making use of Lemma~\ref{lemma:one-dimensional-analytic-set} we see then that the only situation when the set of fixed points of such an endomorphism $G$ is a one-dimensional analytic set, is necessarily the situation when $\alpha=1$. And then the mapping is of the form
\begin{equation}
G(z):=z\cdot p_{\gamma}(h(z)),\; z\in D_{\gamma,r}
\end{equation}
for some $h\in\mathcal O(D_{\gamma,r})$, $h\not\equiv 0$ but with $h^{-1}(0)\neq \emptyset$. Then, we have $\Fix (G)=h^{-1}(0)$. 

Consequently, we get the property that the endomorphism $G$ of degree $1$ defined as above has the set of fixed points being either an analytic set of dimension one (if $h^{-1}(0)\neq\emptyset$, $h\not\equiv 0$), is empty (if $h^{-1}(0)=\emptyset$) or is $D$ (if $h\equiv 0$). 
\end{remark}

Keeping in mind the above observations, we get the following crucial properties of conjugated endomorphisms.

\begin{remark} Let $\psi\colon D_{\gamma_1,r_1}\to D_{\gamma_2,r_2}$ be an iso-conjugating mapping. Let $G$ be an endomorphism of $D_{\gamma_2,r_2}$ of the form:
\begin{equation}
    G(z):= z\cdot p_{\gamma_2}(h(z)),\; z\in D_{\gamma_2,r_2}
\end{equation}
and $h\in\mathcal O(D_{\gamma_2,r_2})$, $h\not\equiv 0$, $h^{-1}(0)\neq\emptyset$.

By Lemma~\ref{lemma:fixed-points} and Remark~\ref{remark:fixed-points-analytic-set} we get that the conjugated endomorphism $F\in\enmo(D_{\gamma_1,r_1})$ has the analogous form
\begin{equation}
F(z):=z\cdot p_{\gamma}(\tilde h(z)),\; z\in D_{\gamma,r},
\end{equation}
where $\tilde h\in\mathcal O(D_{\gamma_1,r_1})$, $\tilde h\not\equiv 0$ and $\Fix (F)=\tilde h^{-1}(0)$. 

Recall the conjugating formula
\begin{equation}
    G\circ\psi=\psi\circ F.
\end{equation}

Note that Lemma~\ref{lemma:fixed-points} gives that  $z\in \Fix(F)$ iff $\psi(z)\in\Fix(G)$ or $\tilde h(z)=0$ iff $h(\psi(z))=0$. 
Exactly the same property holds for the iso-conjugating mapping  $\psi \colon D_{\gamma_1}^*\to D_{\gamma_2}^*$.
\end{remark}

Below, we show the next property of the iso-conjugating mapping of $D_{\gamma,r}$'s and $D_{\gamma}^*$'s.

\begin{lemma}\label{lemma-equivalence-classes} For the iso-conjugating mapping $\psi \colon D_{\gamma_1,r_1}\to D_{\gamma_2,r_2}$ (respectively, $\psi \colon D_{\gamma_1}^*\to D_{\gamma_2}^*$) we have
$\psi([z]_{\sim_{\gamma_1}})=[\psi(z)]_{\sim_{\gamma_2}}$, $z\in D_{\gamma_1,r_1}$ (respectively, $z\in D_{\gamma_1}^*$). Moreover, $\psi$ maps bijectively $[z]_{\sim_{\gamma_1}}=z\cdot p_{\gamma_1}(\mathbb C)$ onto $[z]_{\sim_{\gamma_2}}=\psi(z)\cdot p_{\gamma_2}(\mathbb C)$, $z\in D_{\gamma_1,r_1}$ (respectively, $z\in D_{\gamma_1}^*$).
\end{lemma}
\begin{proof} The proof in both cases is identical. We consider then the case of $D_{\gamma,r}^*$'s. It is sufficient to show that if $z\sim_{\gamma_1} \widetilde{z}$ then $\psi(z)\sim_{\gamma_2}\psi(\widetilde{z})$. Let then $\widetilde{z}=p_{\gamma_1}(\lambda)\cdot z$. Choose $h\in\mathcal O(D_{\gamma_1,r_1})$ with $h\not\equiv 0$, $h(z)=\lambda$, $h^{-1}(0)\neq\emptyset$. We find then $\widetilde{h}\in\mathcal O(D_{\gamma_2,r_2})$ such that 
\begin{equation}
    \psi(\widetilde{z})=\psi(z\cdot p_{\gamma_1}(h(z)))=\psi(z)\cdot p_{\gamma_2}(\widetilde{h}(\psi(z))\in [\psi(z)]_{\sim_{\gamma_2}}. \qedhere
\end{equation}
\end{proof}

%Recall that the endomorphisms of both: $D_{\gamma,r}$'s and $D_{\gamma}^*$'s have images lying in one equivalence class (or more generally $V_{\theta}(\rho)$ for one $\theta$) exactly when their degrees $\alpha=0$. Consequently,
%the iso-conjugation leaves the degree zero endomorphisms invariant.

%\begin{lemma} The iso-conjugation $\psi:D_{\gamma_1,r_1}\to D_{\gamma_2,r_2}$ (respectively, $\psi:D_{\gamma_1}^*\to D_{\gamma_2}^*$) transforms any endomorphism of $D_{\gamma_1,r_1}$ (respectively, $D_{\gamma_1}^*$) of degree zero to an endomorphism of degree zero. Therefore, for any $b\in D_{\gamma_1,r_1}$ (respectively, $b\in D_{\gamma_1}^*$) and any $h\in\mathcal O(D_{\gamma_1,r_1})$ (respectively,  $h\in\mathcal O(D_{\gamma_1,r_1})$) there is a $g\in\mathcal O(D_{\gamma_2,r_2})$ (respectively, $g\in\mathcal O(D_{\gamma_2}^*)$) such that
%\begin{equation}
%\psi(b\cdot p_{\gamma_1}(h(z)))=\psi(b)\cdot p_{\gamma_2}(g(\psi(z))
%\end{equation}
%for $z\in D_{\gamma_1,r_1}$ (respectively, $z\in D_{\gamma_1}^*$).
%\end{lemma}

From now on we denote by $D_j$ either the domains $D_{\gamma_j,r_j}$ or $D_{\gamma_j}^*$ with the iso-conjugating mapping $\psi \colon D_1\to D_2$. Recall that the type of $D_1$ and $D_2$ is the same.

\begin{lemma} Let $\psi \colon D_1\to D_2$ be an iso-conjugating mapping. Then
for any $h\in\mathcal O(D_1)$ there is a $g\in\mathcal O(D_2)$ such that
the endomorphism
\begin{equation}
  F \colon D_1\owns z \mapsto z\cdot p_{\gamma_1}(h(z))\in D_1.    
\end{equation}
is conjugated to
\begin{equation}
    G \colon D_2\owns z \mapsto z\cdot p_{\gamma_2}(g(z)).
\end{equation}
In other words
\begin{equation}
    \psi(z\cdot p_{\gamma_1}(h(z)))=\psi(z)\cdot p_{\gamma_2}(g(\psi(z))), \; z\in D_1.
\end{equation}
\end{lemma}
\begin{proof}
    Let the conjugated endomorphism to $F$ be
\begin{equation}
    G \colon D_2\owns z \mapsto a\cdot \Phi(z)\cdot p_{\gamma_2}(g(z))\in D_2,
\end{equation}
where $\Phi(z):=(z_1^{k_2}z_2^{k_1},z_1^{l_2}z_2^{l_1})$ with $k_j,l_j\in\mathbb Z$, $\alpha\gamma=k_1+\gamma_2 l_1$, $\alpha=k_2+\gamma_2 l_2$ and $a$ is chosen appropriately. We claim that the degree $\alpha=1$. Actually, by Lemma~\ref{lemma-equivalence-classes} we have
\begin{equation}
  \psi(z)\sim_{\gamma_2}\psi(z\cdot p_{\gamma_1}(h(z))=a\cdot \Phi(\psi(z))\cdot p_{\gamma_2}(g(\psi(z))),\; z\in D_1.  
\end{equation}

Taking the expressions $|w_1||w_2|^{\gamma_2}$ in the above equivalence we get
\begin{equation}
    |a_1||a_2|^{\gamma_2}(|\psi_(z)||\psi_2(z)|^{\gamma_2})^{\alpha-1}=1,\; z\in D_1,
\end{equation}
which is only possible if $\alpha=1$. Consequently, $\Phi(z)=z$ and $a\in[(1,1)]_{\sim_{\gamma_2}}$ or $a=p_{\gamma_2}(\lambda_0)$ for some $\lambda_0\in\mathbb C$. Taking $g+\lambda_0$ instead of $g$ we complete the proof.
\end{proof}

A special consequence of the above lemma is that an automorphism conjugated to $z\cdot p_{\gamma_1}(\lambda)$ has degree one and gives the relation
\begin{equation}\label{equation:iso-transformation}
    \psi(z\cdot p_{\gamma_1}(\lambda))=\psi(z)\cdot p_{\gamma_2}(\widehat{B}(\lambda)),
\end{equation}
where $\widehat{B}:\mathbb C\to\mathbb C$ is bijective. Note also that $\widehat{B}(0)=0$. We claim also that $\hat B(\lambda+\mu)=\widehat{B}(\lambda)+\widehat{B}(\mu)$, $\lambda,\mu\in\mathbb C$. Actually, to see the additivity note that by (\ref{equation:iso-transformation}) we get the following equalities
\begin{multline}
\psi(z)\cdot p_{\gamma_2}(\widehat{B}(\mu+\lambda))=\psi(z\cdot p_{\gamma_1}(\mu+\lambda))=   
\psi(z\cdot p_{\gamma_1}(\mu)\cdot p_{\gamma_1}(\lambda))=\\ \psi(z\cdot p_{\gamma_1}(\mu))\cdot p_{\gamma_2}(\widehat{B}(\lambda))=\psi(z)\cdot p_{\gamma_2}(\widehat{B}(\mu))\cdot p_{\gamma_2}(\widehat{B}(\lambda))
\end{multline}
for all $z\in D_1$, $\mu,\lambda\in\mathbb C$.

Altogether for any $h\in\mathcal O(D_1)$ we find a $g\in\mathcal O(D_2)$ such that 
\begin{equation}
    \widehat{B}(h(z))=g(\psi(z)),\; z\in D_1.
\end{equation}

\subsection{Proof of Theorem~\ref{theorem-reinhardt}}
Having all the earlier considerations in mind, we move now to the proof of Theorem~\ref{theorem-reinhardt}.

\begin{proof} We proceed simultaneously with both cases of types of domains.

Recall that the bijective $\widehat{B} \colon \mathbb C\to\mathbb C$ satisfies $\hat{B}(0)=0$, $\hat{B}(\mu+\lambda)=\hat{B}(\mu)+\hat{B}(\lambda)$, $\lambda,\mu\in\mathbb C$. In particular, $\hat{B}(-\lambda)=-\hat{B}(\lambda)$, $\lambda\in\mathbb C$.

We also get the following correspondence induced by $\psi \colon D_1\to D_2$ 
%($D_j$ below is either $D_{\gamma_j}^*$ or $D_{\gamma_j,r_j}$)
\begin{equation}
    \Phi \colon \mathcal O(D_2)\owns g \mapsto \Phi(g):=h\in\mathcal O(D_1)
\end{equation}
that satisfies the relation $\widehat{B}(h(z))=g(\psi(z))$, $z\in D_1$.

This leads to the presentation of $\Phi$ by the formula
\begin{equation}
    \Phi(g)(z)=\hat{B}^{-1}(g(\psi(z))),\; z\in D_1.
\end{equation}
We claim that $\Phi$ is a ring isomorphism. In fact, the additivity follows directly from the additivity of $\hat{B}^{-1}$. To show the multiplicativity, it is sufficient to show that $\Phi(g^2)=\Phi(g)^2$, $g\in\mathcal O(D)$. Denote $h:=\Phi(g)$.

For $z\in D_1$, $\lambda\in\mathbb C$ the following equivalence holds $g(\psi(z))=\hat{B}(\lambda)$ iff $h(z)=\lambda$. In particular, $g(\psi(z))=\widehat{B}(h(z))$, $z\in D$.

Then $g^2(\psi(z))=\hat{B}(\lambda)^2$ iff $g(\psi(z))=\hat{B}(\lambda)$ or $g(\psi(z))=-\hat{B}(\lambda)=\hat{B}(-\lambda)$. The latter is equivalent to: $h(z)=\lambda$ or $h(z)=-\lambda$ which is equivalent to $h^2(z)=\lambda^2$ which means $\Phi(g^2)=h^2$.

Bers' theorem (in the version from \cite{MR0185143}) lets us conclude that there is a(n) (anti-)biholomorphism $\varphi \colon D_1\to D_2$ such that
\begin{equation}\label{equation:conjugation}
    \Phi(g)=g\circ\varphi,\; g\in\mathcal O(D_2)
\end{equation}

Keeping in mind that $D_j$'s are domains of strip type or of     non-complete half plane type by results of \cite{05208635} one may conclude that $D_1=D_2=:D$. And the (anti)-biholomorphism is very regular (in particular, it extends to $\mathbb C_*^2$) -- see Theorem~1 in \cite{05208635}.

%As $\hat B(0)=0$ and $B(1,1)=(1,1)$ we get that $\varphi(1,1)=(1,1)$. So, composing with the automorphism $z^{-1}$ if necessary, the automorphism $\varphi$ is either the identity or the conjugation $\varphi(z)=\overline{z}:=(\overline{z}_1,\overline{z}_2)$. It is sufficient to show that $\psi$ is either the identity or the conjugation, too.

Then
\begin{equation}
    \widehat{B}(g(\varphi(z)))=g(\psi(z)),\; z\in D,\; g\in\mathcal O(D).
\end{equation}
Applying the above to $g=\pi_j$, $j=1,2$ ($\pi_j(z):=z_j$) we get
$\psi(z)=(\hat{B}(\varphi_1(z),\hat{B}(\varphi_2(z)))$, $z\in D$. 

And then
\begin{equation}
    \widehat{B}(g(\varphi(z)))=g(\widehat{B}(\varphi_1(z)),\widehat{B}(\varphi_2(z))),\;g\in\mathcal O(D),\; z\in D.
\end{equation}

But taking the function $g\in\mathcal O(\mathbb C_*)$ as the one defined on $D$ (independent of variable $z_2$) we get
\begin{equation}
    \hat{B}(g(\varphi_1(z)))=g(\hat{B}(\varphi_1(z))),\; g\in\mathcal O(\mathbb C_*),\; z\in D
\end{equation}
or
\begin{equation}
    \hat{B}(g(w_1))=g(\hat{B}(w_1)),\; g\in\mathcal O(\mathbb C_*),\; w_1\in\mathbb C_*,
\end{equation}
which shows by results of Section~\ref{section-punctured-plane} that $\hat{B}$ is (anti)-holomorphic and this finishes the proof.

\end{proof}

\section{A non-pseudoconvex counterexample}
\begin{theorem}
\label{thm:counterexample}
There exist uncountably many holomorphically inequivalent non-pseudoconvex Reinhardt domains $D\subset\mathbb C^2$ such that $\operatorname{End}(D)$ consists only of rotations, i.e.\ mappings $D\owns z\to(\omega_1z_1,\omega_2z_2)\in D$ for $|\omega_j|=1$, and of constant functions.
\end{theorem}

\begin{proof} We start by recalling the endomorphisms of the domain $D_{\gamma}$, where $\gamma>1$ is not algebraic. It consists of the functions $(0,h)$, $(h,0)$ and functions of the form 
\begin{equation}\beta\cdot z^n\cdot p_{\gamma}\circ h,
\end{equation}
where $h\in\mathcal O(D_{\gamma})$, $n=0,1,2,\ldots$, $0<|\beta_1||\beta_2|^{\gamma}\leq 1$, where the last inequality is strict when $n=0$.

Now we define the domain $D$ as follows.
\begin{equation}
    D:=D_{\gamma}\setminus\left((\mathbb C\setminus\mathbb D)\times\{0\}\cup \{0\}\times(\mathbb C\setminus\mathbb D)\cup I_1\cup I_2\right),
    \end{equation}
    where $I_1:=\partial\triangle(0,2)\times \overline{\triangle}(0,1/4^{1/\gamma})$, $I_2:=(\mathbb C\setminus\triangle(0,1/2^{\gamma+1}))\times\partial\triangle(0,2)$.

The set $D$ is the union of $D_{\gamma}^*$ and of two unit discs lying in the axes and with two slits deleted. 

In our case, the essential property is that the logarithmic images of all surfaces $V_r(\gamma)\cap D$ are equal to $V_r(\gamma)$ with one torus removed (for all $r\in(0,1)$, $r\neq 1/2$) or two tori removed (for $r=1/2$).

Note that any holomorphic function on $D$ extends to a holomorphic function on $D_{\gamma}$, so we shall work further with the functions defined on the latter domains. Moreover, one may verify that any endomorphism of $D$ extends to an endomorphism of $D_{\gamma}$. Note also that $D$ is dense in $D_{\gamma}$. Our purpose is to show that if the endomorphism of $D_{\gamma}$ maps $D$ to $D$ then it must be of the form as claimed in the theorem.

We consider first endomorphism of the form $(h,0)$ (and analoguously, $(0,h)$). Note that on $D$ the values of $h$ are bounded by one, so $h$ is bounded by one on $D_{\gamma}$. But any bounded holomorphic function on $D_{\gamma}$ is constant.

Consider the other situation. Note that these endomorphisms (being also endomorphisms of $D_{\gamma}^*$) have the following property: They map the real analytic hypersurface $V_r(\gamma)$ with $0<r<1$ to the surface of the same type $V_s(\gamma)$. 

We prove the following fact that is crucial in the proof of the theorem.

\bigskip

{\bf Claim.} The function $h$ appearing in the formula of an endomorphism must be constant. 

\begin{proof}[Proof of Claim] To show the Claim, it is sufficient to show that for some (equivalently, any) $\beta\in D_{\gamma}^*$ the function
\begin{equation}
 \Phi \colon \mathbb C\owns\lambda \mapsto h(\beta\cdot p_{\gamma}(\lambda))\in\mathbb C
\end{equation}
is constant. To see that this is sufficient, assume $\Phi$ is constant. Then $h$ is constant on $\beta\cdot p_{\gamma}(\mathbb C)$ which, by the irrationality of $\gamma$, is dense in
$V_{|\beta_1||\beta_2|^{\gamma}}(\gamma)$. Consequently, the function $h$ is constant on $V_{|\beta_1||\beta_2|^{\gamma}}(\gamma)$. Then the identity principle shows that $h$ is constant.

Suppose that there is an endomorphism $\Phi(z)=\beta\cdot z^n\cdot p_{\gamma}(h(z))$, $z\in D$ with $h$ being not constant. Choose $r_0\in(0,1)$, $r_0\neq 1/2$. Denote the torus from $V_{r_0}(\gamma)$ not belonging to $D$ by $T$. We claim that $\Phi$ maps $T$ to a torus. To prove that, it is sufficient to show that the mapping $\Phi_0$, where $\Phi_0(z):=\beta\cdot p_{\gamma}(h(z))$, would map $T$ to a torus. If this were the case then the mapping 
\begin{equation}
    D_{\gamma}\owns z \mapsto \log |\beta_2 \exp(h(z)|=\log|\beta_2|+\operatorname{Re}(h(z))
\end{equation}
would be a pluriharmonic mapping that would be constant on the torus $T$. The identity principle would imply that $h$ is constant.

The form of $\Phi$ and the geometry of $D$ let us conclude that the preimages of tori (one or two) from a surface $V_s(\gamma)$ not belonging to the domain $D$ when intersected with some (any) surface $V_r(\gamma)$ must be contained in tori (one or two) from a surface.

Note that by Picard's theorem, for any $a\in T$ the set $\Phi_0(a\cdot p_{\gamma}(\mathbb C))$ touches both tori $T_1$ and $T_2$ that are deleted from the domain $D$. Moreover, to establish the fact that $\Phi_0(T)$ is a torus, it is sufficient to show that if $\Phi_0(a)\in T_1$ then for all $\lambda\in\mathbb C$ with $\operatorname{Re}(\lambda)=0$ we get that $\Phi_0(a\cdot p_{\gamma}(\lambda))\in T_1$. Suppose the opposite. Then, as the mapping
\begin{equation}
    \lambda\mapsto h(a\cdot p_{\gamma}(\lambda))
\end{equation}
is not constant, by the openness of holomorphic functions, we would get the existence of $\lambda$ arbitrarily close to zero with $\operatorname{Re}(\lambda)\neq 0$ and such that $\operatorname{Re}(h(a\cdot p_{\gamma}(\lambda)))=0$, which implies that for $z_0:=a\cdot p_{\gamma}(\lambda)\in V_{r_0}(\gamma)\setminus T$ we get that $\Phi_0(z_0)\in T_1$  -- a contradiction. And this finishes the proof of Claim. 
\end{proof}

%\bigskip

By Claim we know that $\Phi_0(z)=(\beta_1z_1^n,\beta_2z_2^n)$. 

 From the geometry of $D$ (the slits!) we easily get that $n=0$ or $n=1$ (and then $|\beta_1|=|\beta_2|=1$). To see this, one may look at the geometry of $\log D$
 %\begin{equation}
  %   \log D:=\{x\in\mathbb R^n:(e^{x_1},\ldots,e^{x_n})\in D\}
 %\end
 that has to be preserved by the affine mapping $\log D\owns x \mapsto (\log|\beta_1|+n x_1,\log|\beta_2|+nx_2)\in\log D$.
\end{proof}

\section*{Funding}
Rafael Andrist was supported by the European Union (ERC Advanced
grant HPDR, 101053085 to Franc Forstneri\v{c}). 

W\l odzimierz Zwonek was supported by the NCN (National Science Centre, Poland) grant no. 2023/51/B/ST1/01312.

\section*{Conflict of Interest}
The authors have no relevant competing interest to disclose.

\end{document}